\documentclass[fleqn,12pt]{article}
\usepackage{amssymb,amsmath,amsthm}
\usepackage{amsthm}
\usepackage{a4}
\usepackage{amssymb}
\usepackage{amsmath}

\newtheorem*{classics}{\bf Conjecture}
\newtheorem{theorem}{\bf Theorem}[section]
\newtheorem{lemma}{\bf Lemma}[section]
\newtheorem{coro}{\bf Corollary}[section]

\newcommand{\C}{{\mathbb C}}

\newcommand{\Gm}{{\mathbb G}_{\rm m}}

\newcommand{\Z}{{\mathbb Z}}

\newcommand{\Q}{{\Bbb Q}}
\newcommand{\R}{{\mathbb R}}

\newcommand{\V}{{\mathcal V}}
\newcommand{\Hy}{{\mathcal H}}

\newcommand{\bea}{\begin{eqnarray*}}
\newcommand{\eea}{\end{eqnarray*}}
\newcommand{\be}{\begin{eqnarray}}
\newcommand{\ee}{\end{eqnarray}}

\newcommand{\ve}{\boldsymbol}

\newcommand{\spn}{{\rm span}}

\newcommand{\rank}{\mbox{rank}\,}
\newcommand{\lcm}{\mbox{lcm}\,}

\newcommand{\Nc}{N_{\rm tor}}

\begin{document}

\setlength{\unitlength}{1mm}
\setcounter{page}{1}

\title{Solving algebraic equations in roots of unity   
}
\author{Iskander Aliev and Chris Smyth}
%



\maketitle

\begin{abstract}
This paper is devoted to finding solutions of polynomial equations
in roots of unity.  It was conjectured by S. Lang and proved by M.
Laurent that all such solutions can be described in terms of a
finite number of parametric families called maximal torsion
cosets. We obtain new explicit upper bounds for the number of
maximal torsion cosets on an algebraic subvariety of the complex algebraic $n$-torus $\Gm^n$. In contrast to earlier works that give the bounds of polynomial growth in
the maximum total degree of defining polynomials, the proofs of our results are constructive. This allows us to obtain a new algorithm for determining maximal torsion cosets on an algebraic subvariety of $\Gm^n$.


\bigskip
\noindent {\bf 2000 MS Classification:} Primary 11G35; Secondary
11R18.
\end{abstract}

\section{Introduction}

Let $f_1,\ldots,f_t$ be the polynomials in $n$ variables defined
over $\C$. In this paper we deal with solutions of the system
\be \left \{\begin{array}{l} f_1(X_1,\ldots,X_n)=0\,\\ \;\;\vdots \\
f_t(X_1,\ldots,X_n)=0\,
\end{array} \right.\label{Polynomial_system} \ee
in roots of unity. It will be convenient to think of such
solutions as {\em torsion points} on the subvariety ${\mathcal
V}(f_1,\ldots,f_t)$ of the complex algebraic torus $\Gm^n$ defined
by the system (\ref{Polynomial_system}).
%
%
As an affine variety, we identify $\Gm^n$ with the Zariski open
subset $x_1x_2\cdots x_n\neq 0$ of affine space ${\mathbb A}^n$,
with the usual multiplication
\bea (x_1,x_2,\ldots,x_n) \cdot
(y_1,y_2,\ldots,y_n)=(x_1y_1,x_2y_2,\ldots,x_ny_n)\,. \eea
By {\em algebraic subvariety} of  $\Gm^n$ we understand a Zariski
closed subset. An {\em algebraic subgroup} of $\Gm^n$ is a Zariski
closed subgroup. A {\em subtorus} of $\Gm^n$ is a geometrically
irreducible algebraic subgroup. A {\em torsion coset} is a coset
${\ve \omega}H$, where $H$ is a subtorus of $\Gm^n$ and ${\ve
\omega}=(\omega_1, \ldots, \omega_n)$ is a torsion point.
Given an algebraic subvariety ${\mathcal V}$ of $\Gm^n$,  a
torsion coset $C$ is called {\em maximal} in ${\mathcal V}$ if
$C\subset{\mathcal V}$ and it is not properly contained in any
other torsion coset in ${\mathcal V}$. A maximal $0$--dimensional
torsion coset will be also called {\em isolated} torsion point.

Let $\Nc({\mathcal V})$ denote the number of maximal torsion
cosets contained in ${\mathcal V}$.
%
%
A famous conjecture by Lang (\cite{Lang}, p. 221) proved by
McQuillan \cite{McQuillan} implies as a special case that
$\Nc({\mathcal V})$ is finite. This special case had been settled
by Ihara, Serre and Tate (see Lang \cite{Lang}, p. 201) when
$\dim({\mathcal V})=1$,
and by Laurent \cite{Laurent} 
if $\dim ({\mathcal V})> 1$.
A different proof of this result was also given by Sarnak and
Adams \cite{Sarnak}. 
It follows that all solutions of the system
(\ref{Polynomial_system}) in roots of unity can be described in
terms of a finite number of maximal torsion cosets on the
subvariety ${\mathcal V}(f_1,\ldots,f_t)$. It is then of interest
to obtain an upper bound for this number.
Zhang \cite{Zhang_JAMS} and  Bombieri and Zannier
\cite{Bombieri-Zannier}  showed that  if ${\mathcal V}$ is defined
over a number field $K$ then $\Nc({\mathcal V})$ is effectively
bounded in terms of $d$, $n$, $[K:\Q]$ and $M$, when the defining
polynomials  were of total degrees at most $d$ and heights at most
$M$.
Schmidt \cite{Schmidt} found an explicit upper bound for the
number of maximal torsion cosets on an algebraic subvariety of
$\Gm^n$ that depends only on the dimension $n$ and the maximum
total degree $d$ of the defining polynomials. Indeed, let
\bea \Nc(n,d)=\max_{{\mathcal V}}\; \Nc({\mathcal V})\,, \eea
where the maximum is taken over all subvarieties ${\mathcal
V}\subset \Gm^n$ defined by polynomial equations of total degree
at most $d$. The proof of Schmidt's bound is based on a result of
Schlickewei \cite{Schlickewei_AA96} about the number of
nondegenerate solutions of a linear equation in roots of unity.
This latter result was significantly improved by Evertse
\cite{Evertse}, and the resulting Evertse--Schmidt bound can then
be stated as
\be \Nc(n,d)\le (11d)^{n^2}   {{n+d}\choose d}^{3{{n+d}\choose
d}^2}\,.\label{Schmidt_Evertse_bound}\ee

Applying technique from arithmetic algebraic geometry,
David and Philippon \cite{David-Philippon} went even further and obtained a polynomial in $d$
upper bound for the number of isolated torsion points, with the
exponent being essentially $7^k$, where $k$ is the dimension of the
subvariety. This result have been since slightly improved by Amoroso and David \cite{Amoroso-David}.
A polynomial bound for the number of all maximal torsion cosets
also appears in the main result of R\'emond \cite{Remond}, with the exponent  $(k+1)^{3(k+1)^2}$.

It should be mentioned here that the last two bounds are special cases of more general results.
David and Philippon \cite{David-Philippon} in fact study the number of algebraic points with small height
and R\'emond \cite{Remond} deals with subgroups of finite rank and even with
thickness of such subgroups in the sense of the height.
The high generality of the results requires applying sophisticated tools from arithmetic algebraic
geometry. This approach involves work with heights in the fields of algebraic numbers and a delicate specialization
argument (see e. g. Proposition 6.9 in David and Philippon \cite{David-Philippon-2}) that allows to transfer the results to algebraically closed fields of characteristics $0$.

In this paper we present a constructive and more elementary approach to this problem which is based on well--known arithmetic properties of the roots of unity.
Roughly speaking, we use the Minkowski geometry of numbers to reduce the problem to a very special case and then apply an intersection/elimination argument.
This allows us to obtain a polynomial bound  with the exponent $5^n$ for the number of maximal torsion cosets lying on a
subvariety of $\Gm^n$ defined over $\C$ and implies an
algorithm for finding all such cosets. The algorithm is presented in Section \ref{Algorithm}.

One should point out here that the current literature seems to contain only two algorithms for finding all the maximal torsion cosets
on a subvariety of $\Gm^n$. First one was proposed by Sarnak and Adams in \cite{Sarnak} and the second is due to Ruppert \cite{Ruppert}.
Note also that in view of its high complexity, the algorithm of Ruppert is described in \cite{Ruppert} only for a special choice of defining
polynomials.


%
%
%
%

\subsection{The main results} 
We shall start with  the case of
hypersurfaces.
\begin{theorem}
Let $f\in\C[X_1,\ldots,X_n]$, $n\ge 2$, be a polynomial of total
degree $d$ and let $\Hy=\Hy(f)$ be the hypersurface in $\Gm^n$
defined by $f$. Then
\be  \Nc(\Hy)\le c_1(n)\; d^{\;c_2(n)}\,
\label{bound_for_just_one_polynomial}\ee
with 
%
\bea c_1(n)=
n^{\frac{3}{2}(2+n)5^n}\,\;\mbox{and}\,\;\;\;c_2(n)=\frac{1}{16}(49
\cdot 5^{n-2}-4n-9)\,. \eea
\label{just_one_polynomial}
\end{theorem}
%
%
%
%
%
%

Let $f\in\C[X_1,\ldots,X_n]$ be a polynomial of degree $d_i$ in
$X_i$. Ruppert \cite{Ruppert} conjectured that the number of
isolated torsion points on $\Hy(f)$ is bounded by $c(n)\;d_1\cdots
d_n$. Theorem \ref{just_one_polynomial} is a step towards proving
this conjecture.
Furthermore, the results of Beukers and Smyth \cite{Beukers-Smyth}
for the plane curves (see Lemma \ref{BS_main} below) indicate that
the following stronger conjecture might be true.

\begin{classics} The number of isolated torsion points on
the hypersuface $\Hy(f)$ is bounded by $c(n){\rm vol}_n(f)$, where
${\rm vol}_n(f)$ is the $n$-volume of the Newton polytope of the
polynomial $f$.
\end{classics}

Concerning general varieties, we obtained the following result.

\begin{theorem}
For $n\ge 2$ we have
%
\be  \Nc(n,d) \le c_3(n)\; d^{\;c_4(n)} \,,
\label{samaja_glavnaja}\ee
where
\bea c_3(n)= n^{(2+n)2^{n-2}\sum_{i=2}^{n-1}c_2(i)}\prod_{i=2}^n
c_1(i)\,\;\mbox{and}\;\; c_4(n)=\sum_{i=2}^n c_2(i)2^{n-i} +
2^{n-1}\,. \eea
\label{Polynomial_bound}
\end{theorem}
%
%
%
%
%

It should be pointed out that  the constants $c_i(n)$ in Theorems
\ref{just_one_polynomial} and \ref{Polynomial_bound} 
%
could be certainly improved.
To simplify the presentation, we tried to avoid painstaking
estimates.



\subsection{An intersection argument} For ${\ve i}\in\Z^n$, we
abbreviate ${\ve X}^{\ve i}=X_1^{i_1}\cdots X_n^{i_n}$.
Let
\bea f({\ve X})=\sum_{{\ve i}\in \Z^n}a_{\ve i}{\ve X}^{\ve i}\eea
 be a
Laurent polynomial. By the {\em support} of $f$ we mean the set
\bea S_f=\{{\ve i}\in \Z^n: a_{\ve i}\neq 0\}\, \eea
and by the {\em exponent lattice} of $f$ we mean the lattice
$L(f)$ generated by the difference set $D(S_f)=S_f-S_f$, so that
\bea L(f)={\rm span}_{\Z}\{D(S_f)\}\,. \eea
Our next result and its proof is a generalization of that for
$n=2$ in Beukers and Smyth \cite{Beukers-Smyth}.
\begin{theorem}
Let $f\in \C[X_1,\ldots,X_n]$, $n\ge 2$, be an irreducible
polynomial with $L(f)=\Z^n$. Then for some $m$ with $1\le m\le
2^{n+1}-1$ there exist $m$ polynomials $f_1, f_2,\ldots,f_m$ with
the following properties:
\begin{itemize}
\item[{\rm (i)}] $\deg(f_i)\le 2\deg(f)$ for $i=1,\ldots,m$;
%
\item[{\rm (ii)}] For $1\le i\le m$ the polynomials $f$ and $f_i$
have no common factor;

\item[{\rm(iii)}] For any  torsion coset $C$ lying on the
hypersurface $\Hy(f)$ there exists some $f_i$, $1\le i\le m$, such
that
 the coset $C$ also lies on the hypersurface
$\Hy(f_i)$.
\end{itemize}

\label{second_polynomial}\end{theorem}

\section{Lemmas required for the proofs}

In this section, we give the definitions and basic lemmas we need
in the rest of paper.

\subsection{Finding the cyclotomic part of a polynomial in one
variable}

Let us consider the following one-variable version of the problem:
given a polynomial $f\in \C[X]$, find all roots of unity $\omega$
that are zeroes of $f$. This is equivalent to finding the factor
of $f$ consisting of the product of all distinct irreducible
cyclotomic polynomial factors of $f$, which we shall call the {\em
cyclotomic part} of $f$. Algorithms for finding the cyclotomic
part of $f$, using essentially the same ideas, were proposed in
Bradford and Davenport \cite{Bradford-Davenport} and Beukers and
Smyth \cite{Beukers-Smyth}. They are based on the following
properties of roots of unity.
\begin{lemma}[Beukers and Smyth \cite{Beukers-Smyth}, Lemma 1]
\begin{itemize}
\item[{\rm (i)}] If $g\in \C[X]$, $g(0)\neq 0$, is a polynomial
with the  property that for every zero $\alpha$ of $g$, at least
one of $\pm \alpha^2$ is also a zero, then all zeroes of $g$ are
roots of unity. \item[{\rm (ii)}] If $\omega$ is a root of unity,
then it is conjugate to $\omega^p$ where
\bea \left\{\begin{array}{ll} p=2k+1\,,\;\; \omega^p=-\omega &
\mbox{for}\; \omega\, \mbox{a primitive}\; (4k)
\mbox{th root of unity}\,;  \\
p=k+2\,,\;\; \omega^p=-\omega^2 & \mbox{for}\; \omega\, \mbox{a
primitive}\; (2k)
\mbox{th root of unity}\,,\;k\;\mbox{odd}\,;  \\
p=2\,,\;\; \omega^p=\omega^2 & \mbox{for}\; \omega\, \mbox{a}\; k
\mbox{th root of unity}\,,\;k\;\mbox{odd}\,.  \\
\end{array}\right .
\eea
\end{itemize}
\label{basic_properties}
\end{lemma}
%

In the special case $f\in\Z[X]$, Filaseta and Schinzel
\cite{Filaseta-Schinzel} constructed a deterministic algorithm for
finding the cyclotomic part of $f$ that works especially well when
the number of nonzero terms is small compared to the degree of
$f$.

\subsection{Torsion points on plane curves}

Let $f\in \C[X^{\pm 1}, Y^{\pm 1}]$ be a Laurent polynomial. The
problem of finding torsion points on the curve ${\mathcal C}$
defined by the polynomial equation $f(X,Y)=0$  has been addressed
in  Beukers and Smyth \cite{Beukers-Smyth} and Ruppert
\cite{Ruppert}. The polynomial $f$ can be written in the form
\bea f(X,Y)=g(X,Y)\prod_{i} (X^{a_{i}} Y^{b_{i}}-\omega_i)\,, \eea
where the $\omega_j$ are roots of unity and $g$ is a polynomial
(possibly reducible) that has no factor of the form $X^{a}
Y^{b}-\omega$, for $\omega$ a root of unity.
\begin{lemma}[Beukers and Smyth \cite{Beukers-Smyth}, Main Theorem]
The curve ${\mathcal C}$ has at most $22\,{\rm vol}_2(g)$ isolated
torsion points. \label{BS_main}
\end{lemma}
%
%
Hence, for $f\in \C[X, Y]$, the number of isolated torsion points
on the curve ${\mathcal C}=\Hy(f)$ is at most $11\,(\deg(f))^2$.
Furthermore, by Lemma \ref{coset_via_system} below, each factor
$X^{a_{i}} Y^{b_{i}}-\omega_i$ of the polynomial $f$ gives
precisely one torsion coset. Summarizing the above observations,
we get the inequality
\be \Nc({\mathcal C})\le 11(\deg(f))^2+\deg(f)\,.
\label{2D_estimates}\ee

\subsection{Lattices and torsion cosets}
%
We recall some basic definitions. A {\em lattice} is a discrete
subgroup of $\R^n$. Given a lattice $L$ of rank $k$, any set of
vectors $\{{\ve b}_1,\ldots,{\ve b}_k\}$ with $L={\rm
span}_{\Z}\{{\ve b}_1,\ldots,{\ve b}_k\}$ or the matrix ${\bf
B}=({\ve b}_1,\ldots,{\ve b}_k)$ with rows ${\ve b}_i$ will be
called a {\em basis} of $L$.  The {\em determinant} of a lattice
$L$ with a basis ${\bf B}$ is defined to be
\bea  \det (L) = \sqrt{\det({\bf B}\,{\bf B}^T)}\,. \eea
%
%
By an {\em integer} lattice we understand a lattice
$A\subset\Z^n$. An integer lattice is called {\em primitive} if
$A=\spn_{\R}(A)\cap \Z^n$. For an integer lattice $A$, we define
the subgroup $H_A$ of $\Gm^n$ by
\bea H_A=\{{\ve x}\in\Gm^n: {\ve x}^{\ve a}=1\;\; {\rm for\;
all}\;\; {\ve a}\in A\}\,. \eea
Then, for instance, $H_{\Z^n}$ is the trivial subgroup.
%
%
\begin{lemma}[See Schmidt \cite{Schmidt}, Lemmas 1 and 2]
The map $A\mapsto H_A$ sets up a bijection between integer
lattices and algebraic subgroups of $\Gm^n$. A subgroup $H=H_A$ is
irreducible if and only if the lattice $A$ is primitive.
\label{lattices_cosets}
\end{lemma}
%
%
%
 Let ${\ve \omega}=(\omega_1,\ldots,\omega_n)$ be a torsion point and let $C={\ve \omega}H_A$
 be an $r$-dimensional torsion coset with $r\ge 1$.
We will need the following parametric representation of $C$. Let
$\spn_{\R}^\bot(A)$ denote the orthogonal complement of
$\spn_{\R}(A)$ in $\R^n$ and let ${\bf G}=(g_{ij})$ be an
$r\times n$ integer matrix of rank $r$ whose rows ${\ve
g}_1,\ldots,{\ve g}_r$ form a basis of the lattice
$\spn_{\R}^\bot(A)\cap\Z^n$. Then the coset $C$ can be represented
in the form
\bea C=\left(\omega_1\prod_{j=1}^r
t_j^{g_{j1}},\ldots,\omega_n\prod_{j=1}^r
t_j^{g_{jn}}\right)\,\eea
with parameters $t_1, \ldots, t_r\in \C^*$.
We will say that ${\bf G}$ is an {\em exponent matrix} for
the coset $C$.
If $f\in \C[X_1^{\pm 1},\ldots,X_n^{\pm 1}]$ is a Laurent
polynomial and for ${\ve j}\in \Z^r$
\bea f_{\ve j}({\ve X})=\sum_{{\ve i}\in S_f:{\ve i}{\bf
G}^T={\ve j}} a_{{\ve i}}{\ve X}^{\ve i}\,,\eea
then $f({\ve X})=\sum_{{\ve j}\in\Z^r} f_{\ve j}({\ve X})$ and
\be\label{E-fj=0} \text{ the coset } C \text{ lies on } \Hy(f)
\text{ if and only if } f_{\ve j}({\ve \omega})=0 \text{ for all }
{\ve j}\in\Z^r.\ee

Let ${\bf U}=({\ve u}_1, {\ve u}_2,\ldots, {\ve u}_n)$ be a
basis of the lattice $\Z^n$. We will associate with ${\bf U}$
the new coordinates $(Y_1,\ldots,Y_n)$ in $\Gm^n$ defined by
\be Y_1={\ve X}^{{\ve u}_1},\;\; Y_2={\ve X}^{{\ve u}_2},
\ldots,\;\; Y_{n}={\ve X}^{{\ve u}_{n}}\,.
\label{associated_coordinates}\ee
Suppose that the matrix ${\bf U}^{-1}$ has rows ${\ve v}_1,
{\ve v}_2,\ldots, {\ve v}_n$.
%
%
%
%
By the {\em image} of a Laurent polynomial $f\in \C[X_1^{\pm
1},\ldots,X_n^{\pm 1}]$ in coordinates $(Y_1,\ldots,Y_n)$ we mean
the Laurent polynomial
\bea f^{\bf U}({\ve Y})=f({\ve Y}^{{\ve v}_1},\ldots,{\ve
Y}^{{\ve v}_n})\,. \eea
By the {\em image} of a torsion coset  $C={\ve \omega}H_A$ in
coordinates $(Y_1,\ldots,Y_n)$ we mean the torsion coset
\bea C^{\bf U}=({\ve \omega}^{{\ve u}_1},\ldots,{\ve
\omega}^{{\ve u}_n})H_B\,,\eea
where $B=\{{\ve a}{\bf U}^{-1}: {\ve a}\in A\}$.
%
\begin{lemma}
The map $C\mapsto C^{\bf U}$ sets up a bijection between
maximal torsion cosets on the subvarieties ${\mathcal
V}(f_1,\ldots,f_t)$ and ${\mathcal V}(f_1^{\bf
U},\ldots,f_t^{\bf U})$. \label{lattices_cosets}
\end{lemma}
\begin{proof}
It is enough to observe that the map $\phi: \Gm^n\rightarrow
\Gm^n$ defined by
\be \phi({\ve x})=({\ve x}^{{\ve u}_1},\ldots, {\ve x}^{{\ve
u}_{n}})\, \label{auto} \ee
is an automorphism of $\Gm^n$ (see Ch. 3 in Bombieri and Gubler
\cite{Bombieri-Gubler} and Section 2 in Schmidt \cite{Schmidt}).
\end{proof}

\noindent{\bf Remark.} The automorphism (\ref{auto}) is called a
{\em monoidal transformation}. We introduced the coordinates
(\ref{associated_coordinates}) to make the inductive argument used
in the proofs of Theorems
\ref{just_one_polynomial}--\ref{Polynomial_bound} more
transparent.

For $f\in\C[X_1,\ldots,X_n]$ and $k\ge n$, we will denote by
$T^k_i(f)$ the number of $i$-dimensional maximal torsion cosets on
$\Hy(f)$, regarded as a hypersurface in  $\Gm^k$.
Let $A\subset \Z^n$ be an integer lattice of rank $n$ with $\det
(A) > 1$ and let
${\bf A}=({\ve a}_1,\ldots,{\ve a}_n)$ be a basis of $A$.
\begin{lemma}
Suppose that the Laurent polynomials $f, f^* \in \C[X_1^{\pm
1},\ldots,X_n^{\pm 1}]$ satisfy
\be f=f^*({\ve X}^{{\ve a}_1},\ldots,{\ve X}^{{\ve a}_n})\,.
\label{plug_in}\ee
Then the inequalities
\be T^n_i(f^*)\le T^n_i(f)\le \det(A)\,
T^n_i(f^*)\,,\;\;\;i=0,\ldots,n-1\, \label{left-right}\ee
hold. \label{sublattice}
\end{lemma}

\begin{proof}
%
%
%
%
First, for any torsion point ${\ve
\zeta}=(\zeta_1,\ldots,\zeta_n)$ on $\Hy(f^*)$, we will find all
torsion points ${\ve \omega}$ on $\Hy(f)$ with
${\ve \zeta}=({\ve \omega}^{{\ve a}_1},\ldots,{\ve \omega}^{{\ve
a}_n})$.
%
Putting the matrix ${\bf A}$ into Smith Normal Form (see
Newman \cite{Newman}, p. 26) yields two matrices ${\bf V}$
and ${\bf W}$
in ${\rm GL}_n(\Z)$ with ${\bf{WAV}}={\bf D}$, where
${\bf D}=\mbox{diag}(d_1,\ldots,d_n)$. Therefore, by Lemma
\ref{lattices_cosets}, we may assume without loss of generality
that ${\bf A}=\mbox{diag}(d_1,\ldots,d_n)$.
Let $\vartheta_1,\ldots,\vartheta_n$ be primitive $d_1$st,
$d_2$nd, \ldots, $d_n$th roots of $\zeta_1,\ldots, \zeta_n$,
respectively. Then as we let $\vartheta_1,\ldots,\vartheta_n$ vary
over all possible such choices of these primitive roots
\be \begin{array}{l}\mbox{the torsion point}\; {\ve \zeta}\in
\Hy(f^*)\; \mbox{gives precisely}\; \det(A)\; \mbox{torsion}\\
\mbox{points}\;
 {\ve \omega}=(\vartheta_1,\ldots,\vartheta_n)\;\mbox{on}\; \Hy(f)\;\mbox{with}\;{\ve
\zeta}=({\ve \omega}^{{\ve a}_1},\ldots,{\ve \omega}^{{\ve a}_n})
\,.\end{array}\label{preimage}\ee
Let now $M_{f}$ and $M_{f^*}$ denote the sets of all maximal
torsion cosets of {\em positive} dimension on $\Hy(f)$ and
$\Hy(f^*)$ respectively. We will define a map $\tau:
M_f\rightarrow M_{f^*}$ as follows. Let $C\in M_{f}$ be an
$r$-dimensional maximal torsion coset. Given any torsion point
${\ve \omega}=(\omega_1,\ldots,\omega_n)\in C$, we can write the
coset as $C={\ve \omega}H_B$ for some primitive integer lattice
$B$. Recall that $C$ can be also represented in the form
\be C=\left(\omega_1\prod_{j=1}^r
t_j^{g_{j1}},\ldots,\omega_n\prod_{j=1}^r
t_j^{g_{jn}}\right)\,,\label{param}\ee
where $t_1, \ldots, t_r\in \C^*$ are parameters and the vectors
${\ve g}_j=(g_{j1},\ldots,g_{jn})$, $j=1,\ldots, r$, form a basis
of the lattice $\spn_{\R}^\bot(B)\cap\Z^n$. Let $M=\spn_{\Z}\{{\ve
g}_1{\bf A}^T,\ldots,{\ve g}_r{\bf A}^T\}$ and
$L=\spn_{\R}(M)\cap\Z^n$. Then we define
\bea \tau(C)= \left({\ve \omega}^{{\ve a}_1}\prod_{k=1}^r
t_k^{s_{k1}},\ldots,{\ve \omega}^{{\ve a}_n}\prod_{k=1}^r
t_k^{s_{kn}}\right)\,,\eea
%
where $t_1, \ldots, t_r\in \C^*$ are parameters and the vectors
${\ve s}_k=(s_{k1},\ldots,s_{kn})$, $k=1,\ldots, r$, form a basis
of the lattice $L$.
%
%
Let us show that $\tau$ is well-defined.  First, the observation
(\ref{E-fj=0}) implies that $\tau(C)$ is a maximal $r$-dimensional
torsion coset on $\Hy(f^*)$. Now we have to show that $\tau(C)$
does not depend on the choice of ${\ve \omega}\in C$.
Observe that any torsion point ${\ve \eta}\in C$ has the form
\bea {\ve \eta}=\left (\omega_1 \prod_{j=1}^r \nu_j^{
g_{j1}},\ldots, \omega_n \prod_{j=1}^r \nu_j^{ g_{jn}}\right )
\,,\eea
where $\nu_1,\ldots,\nu_r$ are some roots of unity.
%
Put ${\ve h}_j={\ve g}_j {\bf A}^T$, $j=1,\ldots,r$. It is
enough to show that for any roots of unity $\nu_1,\ldots,\nu_r$
there exist roots of unity $\mu_1,\ldots,\mu_r$ such that
\bea \prod_{j=1}^r \nu_j^{h_{ji}}=\prod_{k=1}^r
\mu_k^{s_{ki}}\,,\;\;i=1,\ldots,n\,.\eea
Since $M\subset L$, we have ${\ve h}_j\in L$, so that
\bea {\ve h}_j=l_{j1}{\ve s}_1+\cdots+l_{jr}{\ve
s}_r\,,\;\;\;l_{j1},\ldots, l_{jr}\in \Z\,.\eea
Now we can put
\bea \mu_k=\nu_1^{l_{1k}}\nu_2^{l_{2k}}\cdots
\nu_r^{l_{rk}}\,,\;\;\;k=1,\ldots,r\,.\eea
Thus, the map $\tau$ is well-defined.  It can be also easily shown
that the map $\tau$ is surjective. This observation immediately
implies the left hand side inequality in (\ref{left-right}) for
positive $i$.
Moreover, by (\ref{preimage}), we clearly have
\be T^n_0(f)= \det(A)\, T^n_0(f^*)\,,
\label{equality_for_isolated_points}\ee
so that the lemma is proved for the isolated torsion points.

Let now $D={\ve \zeta}H'\in M^*$ be an $r$-dimensional maximal
torsion coset. Suppose that $D=\tau(C)$ for some $C\in M_f$.
We will show that $C={\ve \omega}H$, where ${\ve \omega}$ can be
chosen among $\det(A)$ torsion points listed in (\ref{preimage}).
This will immediately imply the right hand side inequality in
(\ref{left-right}) for positive $i$.
We may assume without loss of generality that $H=H_B$ and
$H'=H_{\spn^\bot_{\R}(L)\cap \Z^n}$, with the lattices $B$ and $L$
defined as above.
Let $\mu_1,\ldots,\mu_r$ be any roots of unity. Then the coset $D$
can be represented as
\bea D= \left(\zeta_1\prod_{k=1}^r \mu_k^{s_{k1}}\prod_{k=1}^r
t_k^{s_{k1}},\ldots,\zeta_n\prod_{k=1}^r
\mu_k^{s_{kn}}\prod_{k=1}^r t_k^{s_{kn}}\right)\,\eea
%
for ${\ve \zeta}=(\zeta_1,\ldots,\zeta_n)$.
Thus, it is enough to prove the existence of roots of unity
$\nu_1,\ldots,\nu_r$ with
\bea \prod_{k=1}^r \mu_k^{s_{ki}}=\prod_{j=1}^r
\nu_j^{h_{ji}}\,,\;\;i=1,\ldots,n\,.\eea
The lattice $M$ is a sublattice of $L$ and $\rank(M)=\rank(L)$.
Therefore there exist positive integers $n_1,\ldots,n_r$ such that
$n_i {\ve s}_i\in M$, $i=1,\ldots,r$, and, consequently, we have
\bea n_i {\ve s}_i=m_{i1}{\ve h}_1+\cdots+m_{ir}{\ve
h}_r\,,\;\;\;m_{i1},\ldots,m_{ir}\in \Z\,.\eea
Now, if the roots of unity $\rho_1,\ldots,\rho_r$ satisfy
$\rho_i^{n_i}=\mu_i$, $i=1,\ldots, r$, we can put
\bea \nu_j=\rho_1^{m_{1j}}\rho_2^{m_{2j}}\cdots
\rho_r^{m_{rj}}\,,\;\;j=1,\ldots,r\,.\eea

\end{proof}

\subsection{Torsion cosets of codimension one in $\Gm^n$} The next lemma
allows us to detect the $(n-1)$-dimensional torsion cosets on
hypersurfaces.
\begin{lemma}
Suppose that the hypersurface ${\mathcal H}$ is defined by the
polynomial $f\in \C[X_1,\ldots,X_n]$ with $f=\prod_i h_i$, where
$h_i$ are irreducible polynomials. Then the $(n-1)$-dimensional
torsion cosets on ${\mathcal H}$ are precisely the hypersurfaces
$\Hy(h_j)$ defined by the factors $h_j$ of the form ${\ve X}^{{\ve
m}_j}-\omega_j{\ve X}^{{\ve n}_j}$, where $\omega_j$ are roots of
unity. \label{coset_via_system}\end{lemma}
\begin{proof}
Let $\omega$ be a root of unity and let $h={\ve X}^{{\ve
m}}-\omega{\ve X}^{{\ve n}}$ be a factor of $f$.  Multiplying $h$
by a monomial we may assume that $h$ is a Laurent polynomial of
the form ${\ve X}^{\ve a}-\omega$, where ${\ve
a}=(a_1,\ldots,a_n)$ is a {\em primitive} integer vector, so that
$\gcd(a_1,\ldots,a_n)=1$. Let $A$ be the integer lattice generated
by the vector ${\ve a}$,
${\ve b}=(b_1,\ldots,b_n)$ be an integer vector with
$\langle {\ve b},{\ve a}\rangle=1$\,, where $\langle \cdot,
\cdot\rangle$ is the usual inner product,
and put
\bea {\ve \omega}=(\omega^{b_1},\ldots,\omega^{b_n})\,.\eea
 Now, all points of the torsion coset $C={\ve
\omega}H_A$ clearly satisfy the equation ${\ve X}^{{\ve
a}}=\omega$. To show that any solution ${\ve x}=(x_1,\ldots, x_n)$
of this equation belongs to $C$
we observe that the point
$(x_1\omega^{-b_1},\ldots,x_n\omega^{-b_n})$ belongs to the
subtorus $H_A$. 

Conversely, let $C={\ve \omega}H$ be an $(n-1)$-dimensional coset
on ${\mathcal H}$.  Since the exponent matrix of the coset $C$ has
rank $n-1$, there exists a primitive integer vector ${\ve a}$ such
that and for all ${\ve j}\in\Z^{n-1}$ we have $\spn_{\R}(L(f_{\ve
j}))\cap\Z^n=\spn_{\Z}\{{\ve a}\}$. Since $f_{\ve j}({\ve
\omega})=0$, the Laurent polynomial $h_C={\ve X}^{\ve a}-{\ve
\omega}^{\ve a}$ will divide all $f_{\ve j}$ and, consequently,
$f$. Multiplying by a monomial, we may assume that $h_C$ is a
factor of the desired form. Finally, noting that
$H=H_{\spn_{\Z}\{{\ve a}\}}$ and applying the result of the
previous paragraph, we see that $C=\Hy(h_C)$.

\end{proof}

\subsection{Geometry of numbers}   Let $B_p^n$ with $p=1,2,\infty$
denote the unit $n$-ball with respect to the $l_p$-norm, and let
$\gamma_n$ be the Hermite constant for dimension $n$ -- see
Section 38.1 of Gruber--Lekkerkerker \cite{Gruber-Lekkerkerker}.
%
%
For a convex body $K$ and a lattice $L$, we also denote by
$\lambda_i(K,L)$ the $i$th successive minimum of $K$ with respect
to $L$ -- see Section 9.1 ibid.
\begin{lemma}
Let $S$ be a subspace of $\R^n$ with $\dim(S)={\rm
rank}(S\cap\Z^n)=r<n$. Then there exists a basis $\{{\ve b}_1,
{\ve b}_2,\ldots,{\ve b}_n\}$ of the lattice $\Z^n$ such that
\begin{itemize}

\item[{\rm (i)}] $S\subset\spn_{\R}\{{\ve b}_1,\ldots,{\ve
b}_{n-1}\}$;
\item[{\rm (ii)}] $|{\ve b}_i|< 1+
\frac{1}{2}(n-1)\gamma_{n-1}^{\frac{n-1}{2}}
\gamma_{n-r}^{\frac{1}{2}}
%
\det(S\cap\Z^n)^{\frac{1}{n-r}}$, $i=1,\ldots,n$.
\end{itemize}
%
%
\label{suitable_basis}
\end{lemma}

\begin{proof}
Suppose first that $r<n-1$. By Proposition 1 (ii) of Aliev,
Schinzel and Schmidt \cite{Aliev-Schinzel-Schmidt}, there exists a
subspace $T\subset \R^n$ with $\dim(T)= n-1$ such that $S\subset
T$ and
\be \det(T\cap\Z^n) \le
\gamma_{n-r}^{\frac{1}{2}}\det(S\cap\Z^n)^{\frac{1}{n-r}}\,.\label{TviaS}\ee
In the case $r=n-1$ we will put $T=S$.

The subspace $T$ can be considered as a standard
$(n-1)$--dimensional euclidean space.
Then by the Minkowski's second theorem for balls (see Theorem I,
Ch. VIII of Cassels \cite{Cassels}) we have
\bea \prod_{i=1}^{n-1}\lambda_i(T\cap B^n_2, T\cap\Z^n)\le
\gamma_{n-1}^{\frac{n-1}{2}} \det(T\cap\Z^n)\,.\eea
Noting that  $1 \le \lambda_1(T\cap B^n_2, T\cap\Z^n)\le\ldots\le
\lambda_{n-1}(T\cap B^n_2, T\cap\Z^n)$, we get
\be \lambda_{n-1}(T\cap B^n_2, T\cap\Z^n)\le
\gamma_{n-1}^{\frac{n-1}{2}}
\det(T\cap\Z^n)\,.\label{lambda_via_T}\ee
Next, by Corollary of Theorem VII, Ch. VIII of Cassels
\cite{Cassels}, there exists a basis ${\bf B}=({\ve
b}_1,\ldots,{\ve b}_{n-1})$ of the lattice $T\cap\Z^n$ with $|{\ve
b}_j| \le \max\{1, j/2\}\lambda_j(T\cap B^n_2, T\cap\Z^n)$,
$j=1,\ldots,n-1$.
Consequently, 
\bea |{\ve b}_i|\le \frac{n-1}{2}\lambda_{n-1}(T\cap
B^{n}_2,T\cap\Z^n)\le \frac{n-1}{2}\gamma_{n-1}^{\frac{n-1}{2}}
\det(T\cap\Z^n) \eea
\bea \le\frac{n-1}{2}\gamma_{n-1}^{\frac{n-1}{2}}
\gamma_{n-r}^{\frac{1}{2}}\det(S\cap\Z^n)^{\frac{1}{n-r}}\,,\;
%
%
i=1,\ldots,n-1\,.\eea
Further, we need to extend ${\bf B}$ to a basis of the
lattice $\Z^n$. Let ${\ve a}$ be a primitive integer vector from
$\spn^{\bot}_{\R}(T\cap\Z^n)$.
Clearly, all possible vectors ${\ve b}$ such that $({\ve
b}_1,\ldots,{\ve b}_{n-1}, {\ve b})$ is a basis of $\Z^n$ form the
set
$\{{\ve x}\in \R^n: \langle{\ve x},{\ve a}\rangle=\pm
1\}\cap\Z^n$,
and this set contains a point ${\ve b}_n$ with
\be |{\ve b}_n|\le \frac{1}{|{\ve a}|}+\mu(T\cap B^{n}_2,
T\cap\Z^n) \,,\label{via_mu}\ee
where $\mu(\cdot, \cdot)$ is the {\em inhomogeneous minimum} --
see Section 13.1 of Gruber--Lekkerkerker
\cite{Gruber-Lekkerkerker}.
By Jarnik's inequality (see Theorem 1 on p. 99 ibid.)
\bea \mu(T\cap B^{n}_2,T\cap\Z^n) \le
\frac{1}{2}\sum_{i=1}^{n-1}\lambda_i(T\cap B^{n}_2,T\cap\Z^n)\le
\frac{n-1}{2}\lambda_{n-1}(T\cap B^{n}_2,T\cap\Z^n)\,.\eea
Consequently, by (\ref{via_mu}), (\ref{lambda_via_T}) and
(\ref{TviaS}), we have
\bea |{\ve b}_n| 
<1+\frac{n-1}{2}\gamma_{n-1}^{\frac{n-1}{2}}
\gamma_{n-r}^{\frac{1}{2}}\det(S\cap\Z^n)^{\frac{1}{n-r}}\,.\eea
%
%
\end{proof}


When $L$ is a lattice on rank $n$, its {\em polar} lattice $L^*$
is defined as
\bea L^*=\{{\ve x}\in \R^n : \langle {\ve x},{\ve
y}\rangle\in\Z\;\;\mbox{for all}\;\; {\ve y}\in L\}\,. \eea
 Given a basis ${\bf B}=({\ve b}_1,\ldots,{\ve b}_n)$ of
$L$, the basis of $L^*$ {\em polar} to ${\bf B}$ is the basis
${\bf B}^*=({\ve b}_1^*,\ldots,{\ve b}_n^*)$ with
\bea \langle {\ve b}_i\,, {\ve b}_j^*\rangle=\delta_{ij}\,,\;\;
i,j=1,\ldots, n\,,\eea
where $\delta_{ij}$ is the Kronecker delta.
\begin{coro}
Let $S$ be a subspace of $\R^n$ with $\dim(S)={\rm
rank}(S\cap\Z^n)=r<n$. Then there exists a basis ${\bf
A}=({\ve a}_1, {\ve a}_2,\ldots,{\ve a}_n)$ of the lattice $\Z^n$
such that ${\ve a}_1\in S^\bot$ and the vectors of the polar basis
${\bf A}^*=({\ve a}^*_1, {\ve a}^*_2,\ldots,{\ve a}^*_n)$
satisfy the inequalities
\be
|{\ve a}^*_i|< 1+\frac{n-1}{2}\gamma_{n-1}^{\frac{n-1}{2}}
\gamma_{n-r}^{\frac{1}{2}}\det(S\cap\Z^n)^{\frac{1}{n-r}}\,,\;\;\;
i=1,\ldots,n\,.
\label{upper_bound_for_inverse_elements} \ee
\label{suitable_matrix}
\end{coro}
\begin{proof}
Applying Lemma \ref{suitable_basis} to the subspace $S$ we get a
basis $\{{\ve b}_1, {\ve b}_2,\ldots,{\ve b}_n\}$ of $\Z^n$
satisfying conditions (i)--(ii). Observe that its polar basis
$\{{\ve b}^*_1, {\ve b}^*_2,\ldots,{\ve b}^*_n\}$ has its last
vector ${\ve b}^*_n$ in $S^\bot$. Therefore, we can put ${\ve
a}_1={\ve b}^*_n, {\ve a}_2={\ve b}^*_{2},\ldots,{\ve
a}_{n-1}={\ve b}^*_{n-1}, {\ve a}_{n}={\ve b}_1^*$.
\end{proof}

\section{Proof of Theorem \ref{just_one_polynomial}}

%

The lemmas of the next two subsections will allow us to assume
that $L(f)=\Z^n$.

\subsection{$L(f)$ of rank less than $n$}\label{less}

\begin{lemma}
Let $f\in \C[X_1,\ldots,X_n]$, $n\ge 2$, be a 
polynomial of (total) degree $d$. Suppose that $L(f)$ has rank $r$
less
than $n$. Then there exists a 
polynomial $f^*\in
\C[X_1,\ldots,X_{r}]$ of degree at most $d$ such that $L(f^*)$
also has rank $r$ and
\be T^n_i(f)\le T^r_{i-n+r}(f^*)\,,\;\;\;i=n-r,\ldots,n-1\,.
\label{Cosets_on_f_star}\ee
\label{New_polynomial_with_full_lattice}
\end{lemma}
\begin{proof}
Multiplying $f$ by a monomial, we will assume without loss of
generality that $S_f\subset L(f)$. Then there exists an integer
vector ${\ve s}=(s_1,\ldots,s_n)\in \spn_{\R}^\bot(S_f)$ and we
may assume that $s_n\neq 0$. Consider the integer lattice
$A\subset \Z^n$ with the basis
\bea {\bf A}=\left (\begin{array}{lllll} 1&
0&\ldots&0&s_1\\0& 1&\ldots&0&s_2\\ \vdots & \vdots &&\vdots&\vdots\\
0&0&\ldots&1&s_{n-1}\\0&0&\ldots&0&s_n
\end{array}\right )\,. \eea
Observe that
\bea f(X_1,\ldots,X_{n-1},1)=f({\ve X}^{{\ve a}_1},\ldots,{\ve
X}^{{\ve a}_n})\,,
\eea
and, by Lemma \ref{sublattice}, we have
\bea T^n_i(f)\le
T^{n-1}_{i-1}(f(X_1,\ldots,X_{n-1},1))\,,\;\;\;i=1,\ldots,n-1\,.\eea
Applying the same procedure to the polynomial
$f(X_1,\ldots,X_{n-1},1)$ and so on, we will remove $n-r$
variables and get the desired polynomial $f^*$.

\end{proof}

\subsection{$L(f)$ of rank $n$, $L(f)\varsubsetneq\Z^n$}

\begin{lemma}
Let $f\in \C[X_1,\ldots,X_n]$, $n\ge 2$, be an irreducible
polynomial of degree $d$. Suppose that $L(f)$ has rank $n$ and
$L(f)\varsubsetneq\Z^n$. Then there exists an irreducible
polynomial $f^*\in \C[X_1,\ldots,X_n]$ of degree at most
$c_1(n,d)=n^2(n+1)! d$ such that $L(f^*)=\Z^n$ and
\be T^n_0(f)=\det(L(f))T^n_0(f^*)\,,
\label{delicate_step_isolated}\ee
\be T^n_i(f)\le \det(L(f)) T^n_i(f^*)\,,\;\;\;i=1,\ldots,n-1\,.
\label{delicate_step_higher}\ee
\label{New_polynomial_with_Z^n}
\end{lemma}
\begin{proof} Since $S_f\subset d B_1^n$, we have $D(S_f)\subset d
D(B_1^n)=2d B_1^n$. 
Thus, multiplying $f$ by a monomial, we may assume that $f$ is a
Laurent polynomial with $S_f\subset L(f)\cap 2d B_1^n$.
Let $L^*(f)$ be the polar lattice for the lattice $L(f)$ and let
${\bf A}^*=({\ve a}^*_1,\ldots,{\ve a}^*_n)$
be a basis of $L^*(f)$.
Consider the map $\psi: L(f)\rightarrow \Z^n$ defined by
\bea \psi({\ve u}) = (\langle{\ve u},{{\ve a}^*_1}\rangle, \ldots,
\langle{\ve u},{{\ve a}^*_n}\rangle)\,. \eea
The Laurent polynomial
\bea f^*({\ve X})=\sum_{{\ve u}\in \,S_f}a_{\ve u}{\ve
X}^{\psi({\ve u})}\,\eea
has $L(f^*)=\Z^n$.
Observe that we have
\be f=f^*({\ve X}^{{\ve a}_1},\ldots,{\ve X}^{{\ve a}_n})\,.
\label{plug_in}\ee
Therefore the polynomial $f^*$ is irreducible and, by Lemma
\ref{sublattice}, the inequalities (\ref{delicate_step_higher})
hold.  Note also that the equality (\ref{delicate_step_isolated})
follows from (\ref{equality_for_isolated_points}).

Let us estimate the size of $S_{f^*}$.
Recall that  $B_\infty^n$ is the {\em polar reciprocal body} of
$B_1^n$ -- see Theorem III of Ch. IV in Cassels \cite{Cassels}.
%
%
Thus, by Theorem VI of Ch. VIII ibid., we have
\bea \lambda_i(B_1^n, L(f))\lambda_{n+1-i}(B_\infty^n, L^*(f))\le
n!\,. \eea
Noting that $\lambda_i(B_1^n, L(f))\ge 1$, we get the inequality
\be \lambda_{n}(B_\infty^n, L^*(f))\le n!\,.
\label{lambda_star}\ee
Next, by Corollary of Theorem VII, Ch. VIII of Cassels
\cite{Cassels}, there exists a basis ${\bf A}^*=({\ve
a}^*_1,\ldots,{\ve a}^*_n)$ of the lattice $L^*(f)$ such that
\be {\ve a}^*_j\in \max\{1, j/2\} \lambda_j (B_\infty^n, L^*(f))
B_\infty^n \,.\label{a_star}\ee
Combining the inequalities (\ref{lambda_star}) and (\ref{a_star})
we get the bound
\bea ||{\ve a}^*_j||_\infty \le \frac{n\cdot n!}{2}\,.\eea
Then, by the definition of the Laurent polynomial $f^*$, we have
\bea S_{f^*}\subset (\max_{1\le j \le n}||{\ve a}^*_j||_\infty)
2nd B_1^n\subset n^2 n!d B_1^n\,.\eea
Thus, multiplying $f^*$ by a monomial, we may assume that $f^*\in
\C[X_1,\ldots,X_n]$ and
\bea \deg(f^*)\le n^2(n+1)! d=c_1(n,d)\,.\eea

\end{proof}

\subsection{The case $L(f)=\Z^n$}

Let
\bea T(i,n,d)=\max_{\begin{array}{l} \scriptstyle
f\in\C[X_1,\ldots,X_n]\\ \scriptstyle\deg f\le d
\end{array}}T^n_i(f)\,,\;\;\;i=0,\ldots, n-1\, \eea
be the maximum number of maximal torsion $i$-dimensional cosets
lying on a subvariety of $\Gm^n$ defined by a  polynomial of
degree at most $d$.
\begin{lemma} Let $f\in\C[X_1,\ldots,X_n]$, $n\ge 2$, be an irreducible polynomial of degree at most $d$ with
$L(f)=\Z^n$. Then
\be \begin{array}{l}T^n_0(f)\le
(2^{n+1}-1)(T(0,n-1,c_2(n,d))\sum_{s=1}^{n-2}T(s,n-1,2d^2)\\+d
T(0,n-1,2d^2))\,,
\end{array}\label{recurrent_points}\ee
\be \begin{array}{l}T^n_1(f)\le
  (2^{n+1}-1)\,(T(1,n-1,c_2(n,d))
 \sum_{s=1}^{n-2}
T(s,n-1,2d^2)\\ + T(0,n-1,2d^2))\,,
\end{array}\label{recurrent_1_cosets} \ee
\be \begin{array}{l}T^n_i(f)\le (2^{n+1}-1)\,T(i,n-1,c_2(n,d))
 \sum_{s=i-1}^{n-2}
T(s,n-1,2d^2) \,,
\\i=2,\ldots,n-2\,,
\end{array}\label{recurrent_cosets} \ee
\be T^n_{n-1}(f)\le 1\,, \label{large_cosets}\ee
where $c_2(n,d)=n(n+1)d+2(n-1)(n^2-1)n! d^3$.
\label{main_bound}\end{lemma}

\begin{proof}

By Lemma \ref{coset_via_system},  we immediately get the
inequality (\ref{large_cosets}). Assume now that $\Hy(f)$ contains
no $(n-1)$-dimensional cosets.
Applying Theorem \ref{second_polynomial} to the polynomial $f$, we
obtain $m\le 2^{n+1}-1$ polynomials $f_1,f_2,\ldots, f_m$
satisfying conditions (i)--(iii) of this theorem. For $1\le k\le
m$, put $g_k={\rm Res}(f,f_k,X_n)$. By Theorem
\ref{second_polynomial} (ii), the polynomials $f$ and $f_k$ have
no common factor and thus $g_k\neq 0$.
Recall also that  $g_k$ lies in the elimination ideal $\langle f,
f_k\rangle\cap \C[X_1,\ldots, X_{n-1}]$ and $\deg(g_k)\le
\deg(f)\deg(f_k)\le 2d^2$.

Given a maximal $i$--dimensional torsion coset $C$ on $\Hy(f)$,
$i\le n-2$, its orthogonal projection $\pi(C)$ into the coordinate
subspace corresponding to the indeterminates $X_1,\ldots,X_{n-1}$
is a torsion coset in $\Gm^{n-1}$. Note that the coset $\pi(C)$ is
either $i$ or $i-1$ dimensional.
The proof of inequalities
(\ref{recurrent_points})--(\ref{recurrent_cosets}) is based on the
following observation.

\begin{lemma}
Suppose that $1\le k\le m$, $1\le s\le n-2$ and $0\le i\le s+1$.
Then for any maximal torsion $s$-dimensional coset $D$ on the
hypersurface $\Hy(g_k)$ of $\Gm^{n-1}$, the number of maximal
torsion $i$-dimensional cosets $C$ on $\Hy(f)$ with $\pi(C)\subset
D$ is at most $T(i,n-1,c_2(n,d))$.
\label{cosets_in_the_projection}
\end{lemma}

\begin{proof}
Let $D={\ve \omega}H_B$, where $B$ is a primitive sublattice of
$\Z^{n-1}$ with $\rank(B)=n-1-s$. By Corollary
\ref{suitable_matrix}, applied to the subspace
$\spn^{\bot}_{\R}(B)$, there exists a basis ${\bf A}=({\ve
a}_1, \ldots, {\ve a}_{n-1})$ of the lattice $\Z^{n-1}$ such that
${\ve a}_1\in B$ and its polar basis ${\bf A}^*=({\ve a}^*_1,
\ldots, {\ve a}^*_{n-1})$ satisfies the inequality
(\ref{upper_bound_for_inverse_elements}).
Let $C$ be a maximal torsion $i$-dimensional coset on $\Hy(f)$
with $\pi(C)\subset D$. Observe that the coset $D$ and,
consequently, the coset $C$  satisfy the equation
\be (X_1,\ldots, X_{n-1})^{{\ve a}_1}=\omega
\label{system_of_a_coset}\,,\ee
with the root of unity $\omega={\ve \omega}^{{\ve a}_1}$.
The basis ${\bf A}$ of $\Z^{n-1}$ can be extended to the
basis ${\bf B}=(({\ve a}_1,0), \ldots, ({\ve a}_{n-1},0),
{\ve e}_n)$ of $\Z^{n}$, where $({\ve a}_i,0)$ denotes the vector
$(a_{i1},\ldots,a_{i n-1},0)$ and ${\ve e}_n=(0, \ldots,0, 1)$.
Let $(Y_1,\ldots,Y_n)$ be the coordinates associated with
${\bf B}$.
By Lemma \ref{lattices_cosets}, the coset $C^{\bf B}$ is a
maximal $i$--dimensional torsion coset on $\Hy(f^{\bf B})$
and, by (\ref{system_of_a_coset}), 
it lies on the subvariety of $\Hy(f^{\bf B})$ defined by the
equation
$Y_1=\omega$.
%
Therefore, the orthogonal projection of the coset $C^{\bf B}$
into the coordinate subspace corresponding to the indeterminates
$Y_2,\ldots,Y_{n}$ is a maximal $i$--dimensional torsion coset on
the hypersurface $\Hy(f^{\bf B}(\omega,Y_2\ldots,Y_n))$ of
$\Gm^{n-1}$.
Here the polynomial $f^{\bf B}(\omega,Y_{2},\ldots,Y_n)$ is
not identically zero. Otherwise the $(n-1)$-dimensional coset
defined by (\ref{system_of_a_coset}) would lie on the hypersurface
$\Hy(f)$.

The $(n-1-s)$--dimensional subspace $\spn_{\R}(B)$ is generated by
$n-1-s$ vectors of the difference set $D(S_{g_k})$ (see for
instance the proof of Theorem 8 in \cite{McKee-Smyth} for
details). Therefore,
\bea\det(B)\le (\mbox{diam}(S_{g_k}))^{n-1-s}<
(4d^2)^{n-1-s}\,,\eea
where $\mbox{diam}(\cdot)$ denotes the diameter of the set.
It is well known (see e. g. Bombieri and Vaaler
\cite{Bombieri-Vaaler}, pp. 27--28) that
$\det(B)=\det(\spn^{\bot}_{\R}(B)\cap\Z^{n-1})$.
Hence, by (\ref{upper_bound_for_inverse_elements}), we have
\bea S_{f^{\bf B}}\subset (n\max_{1\le j \le n-1}||{\ve
a}^*_j||_\infty) d B_1^n\varsubsetneq
(nd+2n(n-1)\gamma_{n-1}^{\frac{n-1}{2}}
\gamma_{n-1-s}^{\frac{1}{2}}d^3)B_1^n\,.\eea
Multiplying $f^{\bf B}$ by a monomial,  we may assume that
$f^{\bf B}\in \C[Y_1,\ldots,Y_n]$. Now, observing that
$\gamma_k^{k/2}\le k!$, we get
\bea \deg(f^{\bf B})< c_2(n, d)\,.\eea
 Therefore, we have shown that the maximal torsion coset $D$ can contain
projections of at most $T_i^{n-1}(f^{\bf
B}(\omega,Y_2\ldots,Y_n))\le T(i, n-1, c_2(n,d))$ maximal torsion
$i$-dimensional cosets of $\Hy(f)$.

\end{proof}

By part (iii) of Theorem \ref{second_polynomial}, given a maximal
torsion $i$-dimensional coset $C$ on $\Hy(f)$, its projection
$\pi(C)$ lies on $\Hy(g_k)$ for some $1\le k \le m$.
If $i\ge 2$ then the coset $\pi(C)$ has positive dimension, and
Lemma \ref{cosets_in_the_projection} implies the inequality
(\ref{recurrent_cosets}).
Suppose now that $i\le 1$. Let $C$ be a maximal $i$--dimensional
coset on $\Hy(f)$. The case when $\pi(C)$ lies in a torsion coset
of positive dimension of one of the hypersurfaces $\Hy(g_k)$ is
settled by Lemma \ref{cosets_in_the_projection}. It remains only
to consider the case when $\pi(C)$ is an isolated torsion point.
The number of isolated torsion points ${\ve u}$ on $\Hy(f)$ whose
projection $\pi({\ve u})$ is an isolated torsion point on
$\Hy(g_k)$ is at most $dT^{n-1}_0(g_k)\le dT(0,n-1,2d^2)$. Now,
each isolated torsion point on $\Hy(g_k)$ is the $\pi$--projection
of at most one torsion $1$-dimensional coset on $\Hy(f)$.
These observations together with Lemma
\ref{cosets_in_the_projection} imply the inequalities
(\ref{recurrent_points})--(\ref{recurrent_1_cosets}).

\end{proof}

\subsection{Completion of the proof}

Put
$T(n,d)=\sum_{i=0}^{n-1}T(i,n,d)$.
We will show that for $n\ge 2$
\be T(n,d)\le (2nd)^{n+1} T(n-1, n^{8+4n}d^2)T(n-1,
n^{8+4n}d^3)\,. \label{main_recurrency}\ee
This inequality implies Theorem \ref{just_one_polynomial}. Indeed,
noting that, by (\ref{2D_estimates}), we have $T(2,d)\le 11 d^2+d$
and $\Nc(\Hy(f))\le T(n,d)$, we get from (\ref{main_recurrency})
the inequality (\ref{bound_for_just_one_polynomial}).

Let $f\in \C[X_1,\ldots,X_n]$ be a polynomial of degree $d$.
The lattice $L(f)$ clearly has $n$ linearly independent points in
the difference set $D(S_f)$ and $D(S_f)\subset d D(B_1^n)=2d
B_1^n$. Therefore, by Lemma 8 in Cassels \cite{Cassels}, Ch. V,
the lattice $L(f)$ has a basis lying in $ndB_1^n$. Since
$B_1^n\subset B_2^n$, for each irreducible factor $f'$ of $f$ the
inequality
\bea \det(L(f'))\le (nd)^{n}\,\eea
holds. Then, by Lemmas
\ref{New_polynomial_with_full_lattice}--\ref{main_bound} applied
to all irreducible factors of $f$, we have for all $0\le i\le n-1$
\be \begin{array}{ll}T_i^n(f)\le &d  (2^{n+1}-1)(nd)^{n}\times\\
&\times\,T(i,n-1,c_2(n,c_1(n,d))) T(n-1,2(c_1(n,d))^2)\,.\;\;
\end{array}
\label{huge} \ee
To avoid painstaking estimates we simply observe that for $n\ge 3$
and for all $d$ we have $n^{8+4n}d^2>  2(c_1(n,d))^2$ and
$n^{8+4n}d^3> c_2(n,c_1(n,d))$. Then the inequality (\ref{huge})
implies (\ref{main_recurrency}).

\section{Proof of Theorem \ref{Polynomial_bound}}

\begin{lemma} For $n\ge 2$ the inequality
\be \Nc(n,d)\le T(n,d)\Nc(n-1, n^{2+n}d^2)
\label{formula_v_via_t}\ee holds.
\label{v_via_t}
\end{lemma}
\begin{proof}
Suppose that the variety ${\mathcal V}$ is defined by the
polynomials $f=f_1,f_2, \ldots,f_t$.  Then any maximal torsion
coset ${\ve \omega}H$ on ${\mathcal V}$ is contained in a maximal
torsion coset ${\ve \omega}H'$ on the hypersurface $\Hy(f)$.
Now, let $C={\ve \omega}H_A$ with ${\ve
\omega}=(\omega_1,\ldots,\omega_n)$ be a maximal $i$--dimensional
torsion coset on $\Hy(f)$ and suppose $C$ does not lie on
${\mathcal V}$. By Corollary \ref{suitable_matrix}, applied to the
subspace $\spn^{\bot}_{\R}(A)$, there exists a basis ${\bf
A}=({\ve a}_1, {\ve a}_2, \ldots, {\ve a}_{n})$ of the lattice
$\Z^{n}$ such that ${\ve a}_1\in A$ and its polar basis ${\bf
A}^*=({\ve a}^*_1, {\ve a}^*_2, \ldots, {\ve a}^*_{n})$ satisfies
the inequality (\ref{upper_bound_for_inverse_elements}). Let
$(Y_1,\ldots, Y_n)$ be the coordinates associated with the basis
${\bf A}$. By (\ref{associated_coordinates}), the coset
$C^{\bf A}$ lies on the hypersurface of $\Gm^n$ defined by
the equation
\be Y_1=\omega\,, \label{Y_1}\ee
with $\omega={\ve \omega}^{{\ve a}_1}$. Observe that for any
torsion coset ${\ve \zeta}H_B\subset {\ve \omega}H_A$, the lattice
$A$ is a sublattice of the lattice $B$ and ${\ve
\zeta}=(\omega_1x_1,\ldots,\omega_nx_n)$ for some
$(x_1,\ldots,x_n)\in H_A$. Consequently, ${\ve \zeta}H_B$ also
satisfies (\ref{Y_1}).
%
%
Then the number of maximal torsion cosets on ${\mathcal V}$ that
are subcosets of $C$ is at most the number of maximal torsion
cosets on the subvariety of $\Gm^{n-1}$ defined by the 
equations
\bea \begin{array}{l}f_2^{\bf A}(\omega, Y_2,\ldots,Y_n)=0\,,\\
\;\vdots\\ f_t^{\bf A}(\omega,
Y_2,\ldots,Y_n)=0\,.\end{array}\eea
Note that since $C\nsubseteq {\mathcal V}$, not all Laurent
polynomials $f_i^{\bf A}(\omega, Y_2,\ldots,Y_n)$ are
identically zero. 
The $(n-i)$--dimensional subspace $\spn_{\R}(A)$ is spanned by
$n-i$ vectors of the difference set $D(S_f)$. Therefore,
\bea\det(A)\le (\mbox{diam}(S_f))^{n-i}< (2d)^{n-i}\,.\eea
%
%
Note that $\det(A)=\det(\spn^{\bot}_{\R}(A)\cap\Z^n)$.
Hence, by (\ref{upper_bound_for_inverse_elements}), we have
\bea S_{f^{\bf A}_j}\subset d (n\max_{1\le j \le n}||{\ve
a}^*_j||_\infty)  B_1^n \subsetneq
(nd+n(n-1)\gamma_{n-1}^{\frac{n-1}{2}}
\gamma_{n-i}^{\frac{1}{2}}d^2)B_1^n\,\eea
for $j=2,\ldots,t$.  Multiplying the Laurent polynomials
$f_j^{\bf A}$ by a monomial, we may assume that
$f_j^{\bf A}\in \C[Y_2,\ldots,Y_n]$. Noting that
$\gamma_k^{k/2}\le k!$, we get the inequalities
\bea \deg(f^{\bf A}_j)< n(n+1)d+(n-1)(n^2-1)n!
d^2\,,\;\;j=2,\ldots,t\,.\eea
Finally, observe that for $n\ge 2$, $1\le i\le n-1$ and for all
$d$, we have
$$n^{2+n}d^2>n(n+1)d+(n-1)(n^2-1)n! d^2.$$
\end{proof}

By Theorem \ref{just_one_polynomial}, $T(n,d)\le c_1(n)d^{c_2(n)}$
and, consequently,
\bea \Nc(n,d)\le c_1(n)d^{c_2(n)}\Nc(n-1, n^{2+n}d^2)\,.\eea
Noting that $\Nc(1,d)=T(1,d)=d$ we obtain the inequality
(\ref{samaja_glavnaja}).

\section{Proof of Theorem \ref{second_polynomial}}

\subsection{$f$ with rational coefficients}
\label{Rational_coefficients} Suppose that
$f\in\Q[X_1,\ldots,X_n]$, $n\ge 2$, is irreducible and has
$L(f)=\mathbb
Z^n$. We will show that 
 $2^{n+1}-1$ polynomials
\be f(\epsilon_1 X_1, \ldots, \epsilon_n X_n)\,,\;\;\;\epsilon_i=\pm
1\,, \mbox{ not all } \epsilon_i=1 \label{same_degree}\ee
\be f(\epsilon_1 X_1^2, \ldots, \epsilon_n
X_n^2)\,,\;\;\;\epsilon_i=\pm 1\,.\label{twice_degree}\ee
satisfy all conditions of the theorem.

The condition (i) clearly holds for all polynomials
(\ref{same_degree})--(\ref{twice_degree}).
Suppose now that $f$ divides one of the polynomials
(\ref{same_degree}).
Let us consider the lattice
\bea L_2=\left\{(x_1,\ldots,x_n)\in \Z^n:
\frac{1-\epsilon_1}{2}x_1+\ldots+ \frac{1-\epsilon_n}{2}x_n \equiv 0
\mod 2\right\}\, \eea
with the same choice of $\epsilon_i$. Note that $\det(L_2)=2$ and
thus $L_2\varsubsetneq\Z^n$.
Then, for some ${\ve z}\in \Z^n$, we have ${\ve z}+S_f\subset
L_2$. Therefore the lattice $L(f)$ cannot coincide with $\Z^n$, a
contradiction. This argument also implies that the polynomials
(\ref{same_degree}) are pairwise coprime.
Next, if $f$ divides a polynomial $f'$ from (\ref{twice_degree})
then, since $f'\in\Q[X_1^2,\ldots,X_n^2]$, we have that each of
the polynomials (\ref{same_degree}) also divides $f'$. Hence
$2^n\deg f\le \deg f'=2\deg f$, so that $n=1$, a contradiction.
Consequently, the set of polynomials $f_1,\ldots,f_m$ consists of
all the polynomials (\ref{same_degree})--(\ref{twice_degree}).
Then condition (ii) is satisfied.

It remains only to check that the condition (iii) holds. Let
$C={\ve \omega}H$ be a torsion $r$-dimensional coset on the
hypersurface
$\Hy=\Hy (f)$. 
%
There is a root of unity
$\omega$ such that ${\ve \omega}=(\omega^{i_1}, \ldots,
\omega^{i_n})$,
%
%
where we may assume that $\gcd (i_1,\ldots,i_n)=1$ so that,  in
particular,  not all of the $i_1,\ldots,i_n$ are even. Next, we have
\bea f(\omega^{i_1}, \ldots, \omega^{i_n})=0\,\eea
and by part (ii) of Lemma \ref{basic_properties}, also at least
one of the $2^{n+1}-1$ equalities
\bea f(\epsilon_1 \omega^{i_1}, \ldots, \epsilon_n
\omega^{i_n})=0\,,\;\;\;\epsilon_i=\pm 1\,, \mbox{ not all }
\epsilon_i=1 \eea
\bea f(\epsilon_1 \omega^{2i_1}, \ldots, \epsilon_n
\omega^{2i_n})=0\,,\;\;\;\epsilon_i=\pm 1\eea
holds.
Therefore, the torsion point ${\ve \omega}$ lies on a hypersurface
$\Hy'=\Hy(f')$, where $f'$ is one of the polynomials
$f_1,\ldots,f_m$. This settles the case $r=0$.

Suppose now that $r\ge 1$.  We claim that the torsion coset $C$
lies on $\Hy'$. To see this we observe that for all ${\ve
j}\in\Z^r$ we have
\bea f'_{\ve j}({\ve \omega})=f_{\ve j}(\omega^{p\, i_1}, \ldots,
\omega^{p\, i_n})=0\,,\eea
where $p$ is the exponent from the part (ii) of Lemma
\ref{basic_properties}. Hence by (\ref{E-fj=0}), $C$ lies on
$\Hy'$.

\subsection{$f$ with coefficients in $\Q^{\rm ab}$}

We now define the polynomials $f_1,\ldots,f_m$ in the case of $f$
having coefficients lying in a cyclotomic field.
Let us choose $N$
 to be
the smallest integer such that, for some roots of unity
$\zeta_1,\ldots, \zeta_n$, the polynomial $f(\zeta_1
x_1,\ldots,\zeta_n x_n)$ has all its coefficients in
$K=\Q(\omega_N)$, for $\omega_N$ a primitive $N$th root of unity.
Since for $N$  odd $-\omega_N$ is a primitive $(2N)$th root of
unity, we may assume either that $N$ is odd or a multiple of $4$.

We then replace $f$ by this polynomial. When we have found the
polynomials $f_1,\ldots,f_m$ for this new $f$, it is easy to go back
and find those for the original $f$.


\subsubsection{$N$ odd}

Take $\sigma$ to be an automorphism of $K$ taking $\omega_N$ to
$\omega_N^2$. We keep the polynomials $f_i$ that come from
(\ref{same_degree}) and replace the polynomials that come from
(\ref{twice_degree}) by
\be f^\sigma(\epsilon_1 X_1^2, \ldots, \epsilon_n
X_n^2)\,,\;\;\;\epsilon_i=\pm 1\,,\;\;\;\mbox{not divisible by
}\;f\,.\label{twice_degree_sigma}\ee
We then claim that any torsion coset of $\Hy(f)$ either lies on
one of the $2^n-1$ hypersurfaces defined by (\ref{same_degree}) or
on one of the $2^n$ hypersurfaces defined by one of the
polynomials (\ref{twice_degree_sigma}). Take a torsion coset
$C=(\omega_l^{i_1},\ldots,\omega_l^{i_n})H$ of $\Hy(f)$, with
$\gcd(i_1,\ldots,i_n)=1$. If $4\nmid l$ then we can extend
$\sigma$ to an automorphism of $K(\omega_l)$ which takes
$\omega_l$ to one of $\pm \omega_l^2$. Therefore, the coset $C$
also lies on a hypersurface defined by one of the polynomials
(\ref{twice_degree_sigma}). On the other hand, if $4|l$, we put
$4k=\lcm(l,N)$. Then the automorphism, $\tau$ say, of
$K(\omega_l)=\Q(\omega_{4k})$ mapping
$\omega_{4k}\mapsto\omega_{4k}^{2k+1}$ takes $\omega_l\mapsto
\omega_l^{2k+1}=-\omega_l$ and $\omega_N\mapsto
\omega_N^{2k+1}=\omega_N$. Thus, $C$ lies on a hypersurface
defined by one of the polynomials (\ref{same_degree}).

\subsubsection{$4|N$}

We take the same coset $C$ as in the previous case, again put
$4k=\lcm(l,N)$, and use the same automorphism $\tau$. Then $\tau$
takes $\omega_l\mapsto\omega_l^{2k}\omega_l=\pm\omega_l$ and
$\omega_N\mapsto\omega_N^{2k}\omega_N=\pm\omega_N$. We now consider
separately the four possibilities for these signs. Firstly, from the
definition of $k$ they cannot both be $+$ signs.

If
\bea \tau(\omega_l)=\omega_l\,,\;\;\; \tau(\omega_N)=-\omega_N\eea
then $C$ also lies on $\Hy(f^\tau)$. Note that $f^\tau\neq f$, by
the minimality of $N$, so that they have a proper intersection.

If
\bea \tau(\omega_l)=-\omega_l\,,\;\;\; \tau(\omega_N)=\omega_N\eea
then $C$ also lies on a hypersurface defined by one of the
polynomials (\ref{same_degree}). As $L(f)=\Z^n$, each has proper
intersection with $f$, as we saw in Section
\ref{Rational_coefficients}.

Finally, if
\bea \tau(\omega_l)=-\omega_l\,,\;\;\;
\tau(\omega_N)=-\omega_N\eea
then $C$ also lies on one of the hypersurfaces $\Hy(f_i^\tau)$, for
$f_i$ in (\ref{same_degree}). Suppose that for instance $f$ and
$f^\tau(-X_1,X_2,\ldots,X_n)$ have a common component, so that
$f^\tau(-X_1,X_2,\ldots,X_n)=f(X_1,X_2,\ldots,X_n)$. Then we have
\bea f(\omega_N X_1,X_2,\ldots,X_n)^\tau=f^\tau(-\omega_N
X_1,X_2,\ldots,X_n)=f(\omega_N X_1,X_2,\ldots,X_n)\,.\eea
For any coefficient $c$ of $f(\omega_N X_1,X_2,\ldots,X_n)$, write
$c=a+\omega_N$b, where $a,b\in \Q(\omega^2_N)$. Then
$c^\tau=a-\omega_N b=c$, so that $b=0$, $c\in\Q(\omega^2_N)$.
Consequently, $f(\omega_N X_1,X_2,\ldots,X_n)\in
\Q(\omega_N^2)[X_1,\ldots,X_n]$, contradicting the minimality of
$N$.
The same argument applies for other polynomials
(\ref{same_degree}).
Thus, $C$ lies on one of $2^{n+1}-1$ subvarieties defined by the
polynomials (\ref{same_degree}) and the polynomials
\bea f^\tau(\epsilon_1 X_1, \ldots, \epsilon_n
X_n)\,,\;\;\;\epsilon_i=\pm 1\,. \eea

\subsection{$f$ with coefficients in $\C$}

Let  $L$ be the coefficient field of $f$. Suppose that $L$ is not
a subfield of $\Q^{\rm ab}$. Without loss of generality, assume
that at least one coefficient of $f$ is equal to $1$ and choose an
automorphism $\sigma\in \mbox{Gal}(L/\Q^{\rm ab})$ which does not
fix $f$. Then since all roots of unity belong to $\Q^{\rm ab}$,
$f$ and $f^\sigma$ have the same torsion cosets. Further, $f$ and
$f^\sigma$ have no common component. Thus in this case we can take
the set of $f_i$ to be the single polynomial $f^\sigma$.

\section{The algorithm}
\label{Algorithm}

Let ${\mathcal V}$ be an algebraic subvariety of $\Gm^n$. In this
section we will describe a new recursive algorithm that finds all
maximal torsion cosets on ${\mathcal V}$. The algorithm consists
of several reduction steps that reduce the problem to finding
maximal torsion cosets of a finite number of subvarieties of
$\Gm^{n-1}$. When $n=2$ we can apply the algorithm of Beukers and
Smyth \cite{Beukers-Smyth}.

 \subsection{Hypersurfaces} We first consider a hypersurface $\Hy$ defined by a polynomial
$f\in \C[X_1,\ldots,X_n]$ with $f=\prod h_i$, where $h_i$ are
irreducible polynomials. By Lemma \ref{coset_via_system}, the
$(n-1)$-dimensional torsion cosets on $\Hy$ will precisely
correspond to the factors $h_j$ of the form ${\ve X}^{{\ve
u}_j}-\omega_j{\ve X}^{{\ve v}_j}$, where $\omega$ is a root of
unity.
Now we will assume without loss of generality that $f$ is
irreducible and $\Hy$ contains no torsion cosets of dimension
$n-1$. Then we proceed as follows.
\begin{itemize}

\item[H1.] The proofs of Lemmas
\ref{New_polynomial_with_full_lattice},
\ref{New_polynomial_with_Z^n} and Theorem \ref{second_polynomial}
are effective. Consequently, applying Lemmas
\ref{New_polynomial_with_full_lattice} and
\ref{New_polynomial_with_Z^n}, we may assume without loss of
generality that $L(f)=\Z^n$. Next, applying
 Theorem \ref{second_polynomial}, we get $m<2^{n+1}$ polynomials
$f_1,\ldots,f_m$ satisfying conditions (i)--(iii) of this theorem.

\item[H2.] For $1\le k\le m$, calculate $g_k={\rm
Res}(f,f_k,X_n)$. Find all isolated torsion points ${\ve \zeta}_1,
{\ve \zeta}_2, \ldots$ and all maximal torsion cosets $D_1, D_2,
\ldots$ of positive dimension on the hypersurfaces $\Hy(g_k)$ of
$\Gm^{n-1}$. For each coset $D_i={\ve \eta}_iH_{B_i}$, take a
primitive vector ${\ve a}_i\in B_i$ and put $\omega_i={\ve
\eta}_i^{{\ve a}_i}$.

\item[H3.] For each torsion point ${\ve
\zeta}_i=(\zeta_{i1},\ldots,\zeta_{i\,n-1})$, if
$f(\zeta_{i1},\ldots,\zeta_{i\,n-1}, X_n)$ is identically zero
then the coset
\bea (\zeta_{i\,1},\ldots,\zeta_{i\,n-1},t)\eea
lies on ${\mathcal H}$. Otherwise, solving the polynomial equation
$f(\zeta_{i\,1},\ldots,\zeta_{i\,n-1}, X_n)$ in $X_n$, we will
find all torsion  points ${\ve \zeta}$ on $\Hy$ with $\pi({\ve
\zeta})={\ve \zeta}_i$. When all torsion cosets of positive
dimension on $\Hy$ are found, we can easily determine which of the
torsion points ${\ve \zeta}$ are isolated.

\item[H4.] For each $D_i$, extend the vector ${\ve a}_i$ to a
basis ${\bf B}_i=(({\ve a}_i,0),{\ve z}_2,\ldots, {\ve
z}_{n})$ of $\Z^n$. Find all maximal torsion cosets $E_1, E_2,
\ldots$ on the hypersurface in $\Gm^{n-1}$ defined by the
polynomial $f^{{\bf B}_i}(\omega_i, Y_2,\ldots,Y_n)$. For
each $E_j={\ve \rho}_jH_{P_j}$ say with ${\ve
\rho}_j=(\rho_{j\,2},\ldots,\rho_{j\,n})$ put ${\ve
\omega}_j=(\omega_i, \rho_{j\,2},\ldots,\rho_{j\,n})$ and
$A_j=\{(z, p_2,\ldots,p_n): z\in \Z\,,\;(p_2,\ldots,p_n)\in
P_j\}$. Now the cosets $({\ve \omega}_jH_{A_j})^{{\bf
B}_i^{-1}}$ are the maximal torsion cosets on $\Hy$.

\end{itemize}

\subsection{General subvarieties} Suppose now that ${\mathcal V}$ is defined by the
polynomials $f_1,\ldots,f_t\in \C[X_1,\ldots,X_n]$, when $t\ge 2$.

\begin{itemize}

\item[V1.] Find all isolated torsion points ${\ve \zeta}_1, {\ve
\zeta}_2, \ldots$ and all maximal torsion cosets $D_1, D_2,
\ldots$ of positive dimension on the hypersurface $\Hy(f_1)$. Then
${\ve \zeta}_1, {\ve \zeta}_2, \ldots$, if on $\V$, are isolated
torsion points on $\V$ as well.

\item[V2.] For each coset $D_i={\ve \eta}_iH_{B_i}$, take a
primitive vector ${\ve a}_i\in B_i$, put $\omega_i={\ve
\eta}_i^{{\ve a}_i}$ and extend the vector ${\ve a}_i$ to a basis
${\bf B}_i=({\ve a}_i,{\ve z}_2,\ldots, {\ve z}_{n})$ of
$\Z^n$. Find all maximal torsion cosets $E_1, E_2, \ldots$ on the
subvariety of $\Gm^{n-1}$ defined by the polynomials
$f_k^{{\bf B}_i}(\omega_i, Y_2,\ldots,Y_n)$, $k=2,\ldots, t$.
For each $E_j={\ve \rho}_jH_{P_j}$ with ${\ve
\rho}_j=(\rho_{j\,2},\ldots,\rho_{j\,n})$ put ${\ve
\omega}_j=(\omega_i, \rho_{j\,2},\ldots,\rho_{j\,n})$ and
$A_j=\{(z, p_2,\ldots,p_n): z\in \Z\,,\;(p_2,\ldots,p_n)\in
P_j\}$. Now the cosets $({\ve \omega}_jH_{A_j})^{{\bf
B}_i^{-1}}$, along with the isolated torsion points found in step
V1, are the maximal torsion cosets on $\V$.

\end{itemize}

The described algorithm clearly stops after a finite number of
steps and the proofs of Theorems \ref{just_one_polynomial} and
\ref{Polynomial_bound} show that the algorithm finds all maximal
torsion cosets on $\V$. Furthermore, the constants $c_i(n,d)$ give
explicit bounds for the degrees of the polynomials generated at
each step.

\section{Acknowledgement}

The authors are very grateful to Professors Patrice Philippon and Andrzej Schinzel for
important comments and to Doctor Tristram De Piro for helpful
discussions.

School of Mathematics and Wales Institute of Mathematical and Computational Sciences, Cardiff University, Senghennydd Road, Cardiff CF24 4AG UK

{\em E-mail address:} alievi@cf.ac.uk

School of Mathematics and Maxwell Institute for Mathematical
Sciences, University of Edinburgh, Kings Buildings, Edinburgh EH9
3JZ UK

{\em E-mail address:}  C.Smyth@ed.ac.uk

\enddocument